\documentclass[12pt, oneside]{amsart}       % use "amsart" instead of "article" for AMSLaTeX format
\usepackage{geometry}   
\usepackage[T1]{fontenc}
\usepackage{palatino}
\usepackage{graphicx}                % Use pdf, png, jpg, or epsÂ§ with pdflatex; use eps in DVI mode
                            \usepackage{amsthm}   
\usepackage{amsaddr}
\newtheorem{theorem}{Theorem}[section]

\newtheorem{definition}{Definition}[section]
\newtheorem{lemma}[theorem]{Lemma}
\theoremstyle{remark}
\newtheorem*{remark}{Remark}
\usepackage{tcolorbox}
\newtheorem*{example}{Example}

\title{The Chiral critical locus and Topological Structures}
\author{Emile Bouaziz}

\email{emile.g.bouaziz@gmail.com}
\begin{document}\begin{abstract} We study a differential graded VOA associated to the derived critical locus of a function $f$ on a smooth oriented $D$-dimensional variety $(X,\mathbf{vol})$. Informally, this VOA, $\mathbf{crit}^{ch}_{f}$, is just the algebra of chiral differential operators on the derived critical locus $\mathbf{crit}_{f}$. We prove, using a generalization of a physical construction of Witten, the $\mathbf{crit}^{ch}_{f}$ admits a \emph{topological structure} if $f$ is homogeneous for a $\mathbf{G}_{m}$ action on $(X,\mathbf{vol})$. If $\mathbf{vol}$ has weight $b$ and $f$ has weight $a$, we compute the rank of the topological structure in terms of the discrete invariants of the theory to be $$d=\Big(D-\frac{2b}{a}\Big).$$ We conclude with some remarks about BV quantization and a simple computation of characters. \end{abstract}\maketitle

\section{Introduction and Statement of Results}

This note deals with some rather simple \emph{chiral calculus}, in the sense of \cite{MSV} and more recently\cite{MST} and \cite{Bo}. It is by now well understood that many spaces of interest naturally arise as \emph{derived} critical loci of functions on smooth varieties. On physical grounds, this is natural from a variational perspective. 

Our goal in this note is to study some vertex theoretic (alias \emph{chiral}) invariants attached to critical loci. Special attention is paid to the algebras of symmetries of such invariants. 

We will attempt to sketch our main ideas, leaving proper explanations of the definitions and objects in play for later sections. We let $(X,\mathbf{vol})$ denote a smooth $D$-dimensional oriented algebraic variety over $\mathbf{C}$, ie we endow $X$ with a choice of volume form $\mathbf{vol}\in\Omega^{D}_{X}$. Then a highly influential result of \cite{MSV} states that $(X,\mathbf{vol})$ determines a sheaf on $X$ of \emph{vertex operator algebras} (VOAs henceforth) equipped with an action of the $\mathbf{top}\{D\}$, the \emph{topological vertex algebra} at \emph{rank} $D$. The resulting sheaf of vertex algebras in denoted $\Omega^{ch}_{X}$. It exists without requiring a volume form, but the topological structure depends on having chosen $\mathbf{vol}$. We will be primarily interested in a slight variant of the sheaf of VOAs $\Omega^{ch}_{X}$, denoted $\Theta^{ch}_{X}$ and referred to as \emph{chiral polyvectors} - this is defined (see below) as $\Omega^{ch}_{X}$ with topological structure twisted by the \emph{mirror involution} of $\mathbf{top}\{D\}$. 

With the above in place, we define our main object of study, the \emph{chiral critical locus}, denoted $\mathbf{crit}^{ch}_{f}$. If $G$ is the odd conformal weight $2$ generator of $\mathbf{top}\{D\}$, then we define $$\mathbf{crit}^{ch}_{f}:=\operatorname{BRST}(\Theta^{ch}_{X},G_{(0)}f).$$  That is to say, we define the differential graded VOA $$\mathbf{crit}^{ch}_{f}:=\big(\Theta^{ch}_{X},\partial_{f}^{ch}:=(G_{(0)}f)_{(0)}\big).$$ 
We remark that $\partial^{ch}_{f}$ is the $0$th vertex Gerstenhaber bracket with $f$, in the sense of \cite{Bo}. In particular it restricts to $\{f,-\}$ on the conformal weight $0$ subspace $$\Theta^{\operatorname{ch},(0)}_{X}\cong\bigwedge\Theta_{X},$$ where $\{-,-\}$ denotes the bracket of polyvectors and $\Theta_{X}$ the tangent sheaf. In particular we have $$\mathbf{crit}^{\operatorname{ch},(0)}_{f}\cong\big(\bigwedge\Theta_{X},\{f,-\}\big)\cong\mathcal{O}(\mathbf{crit}_{f}),$$ justifying our choice of name. We will show that $\mathbf{crit}^{ch}_{f}$ admits large algebra of symmetries, if $f$ and $\mathbf{vol}$ are homogeneous for a $\mathbf{G}_{m}$ action on $X$. More precisely, we will show that it admits the structure of a \emph{topological vertex algebra} at a rank depending on the dimension of $X$ and the weights of $\mathbf{vol}$ and $f$. This discussion is summarized in the theorem below, which is proven in physical language in \cite{Witt} in the special case of $(\mathbf{A}^{1},\mathbf{vol}=dx,f=x^{a})$, and for products of this special case.

\begin{theorem} Let $f$ be homogeneous of weight $a$ for a $\mathbf{G}_{m}$ action on $(X,\mathbf{vol})$, with $\mathbf{vol}$ of weight b. Then there are $\partial_{f}^{ch}$-closed currents $\{^{f}L,\,^{f}J,\,^{f}Q,\,^{f}G\}$, with $^{f}L=L,\,^{f}G=G,$ generating an action of $\mathbf{top}\{d\}$ on $\mathbf{crit}^{ch}_{f}$, with $$d:=\Big(D-\frac{2b}{a}\,\Big).$$\end{theorem}

We will conclude in section 4 with some unsurprising computations of the cohomology of the natural BV quantization of $\mathbf{crit}^{ch}_{f}$, and of its characters. 

\section{Preparatory Notions}

The initiated reader can safely skip this section, referring to it as necessary for notational conventions etc.

\subsection{Vertex Algebras} We will not give a proper account of the theory of vertex algebras in this note, and instead we essentially just fix some notation. Recall that a vertex algebra $\mathbf{V}$ consists of a vector space $\mathbf{V}$, a distinguished \emph{vacuum vector} $\Omega$, an operator of \emph{infinitesimal translation}, $\delta:\mathbf{V}\rightarrow \mathbf{V}$ and a \emph{state-field correspondence} $$A\mapsto A(z)=\sum_{n}A_{(n)}z^{-1-n}.$$ Here $A_{(n)}$ are endomorphisms on $\mathbf{V}$, so that for all $A,B\in\mathbf{V}$, $A(z)B\in \mathbf{V}((z))$. The above structure is further subject to some axioms, for which the reader can consult \cite{Kac}. A consequence of the axioms is  \emph{Borcherds' formula} for commutators of fields. We set $\delta^{(i)}(z/w):=\partial_{w}^{i}\sum_{n}z^{n}w^{-1-n}$ in what follows. Then we have the following formula due to Borcherds, where the sum over $n$ is necessarily finite - $$[A(z),B(w)]=\sum_{n}A_{(n)}B(w)\frac{1}{n!}\delta^{(n)}(z/w).$$ We follow the common practice of expressing the above symbolically as the \emph{operator product expansion} (OPE) equation - $$A(z)B(w)\sim\sum\frac{A_{(n)}B(w)}{(z-w)^{n+1}}.$$ If $A,B\in\mathbf{V}$, we will write $AB:=A_{(-1)}B$, with the caution that this is neither associative nor commutative in general.

 Recall further that a vertex algebra is called a vertex operator algebra or VOA if there is an element $L\in\mathbf{V}$ so that writing $L(z)=\sum L_{n}z^{-2-n}$, the $L_{i}$ satisfy the commutation relations of the Virasoro Lie algebra at some central charge $c$. We further demand that $L_{(0)}=L_{-1}=\delta$ and finally that $L_{(1)}=L_{0}$ is diagonalizable. In all cases of interest to us it will in fact have eigenvalues positive integers. 

\begin{remark} In terms of the field $L(z)$, the above relations on the modes amount to the OPE $$L(z)L(w)\sim\frac{\partial L(w)}{(z-w)}+\frac{2L(w)}{(z-w)^{2}}+\frac{c}{2(z-w)^{4}}.$$\end{remark}

\begin{remark} The above considerations generalize verbatim to the super context, where commutators are taken in the super sense. There is also a notion of a ($\mathbf{Z}$-) graded VOA, for which the reader can refer to \cite{MSV}. When we refer to the \emph{degree} of an element of a graded VOA, it is with respect to the defining grading. We will reserve the word \emph{weight} for eigenvalues with respect to $L_{(1)}$. \end{remark}

\subsection{BRST Reduction}

We build differential graded VOAs by a process called \emph{BRST reduction}, which we recall here.
\begin{definition} An element $A\in \mathbf{V}$ is called a Virasoro primary of weight $a$ if we have the following OPE: $$L(z)A(w)\sim\frac{\delta A(w)}{(z-w)}+\frac{aA(w)}{(z-w)^{2}}.$$\end{definition}

\begin{lemma} If $A$ is a Virasoro primary of conformal weight $1$ in a VOA with Virasoro $L$, then we have $A_{(0)}L=0$.\end{lemma}

\begin{proof} This is immediate from quasi-symmetry of OPEs, see \cite{Kac}. \end{proof}

\begin{definition}An operator $s:\mathbf{V}\rightarrow \mathbf{V}$ is called a \emph{derivation} of $\mathbf{V}$ if for all $A\in\mathbf{V}$ we have $$[s,A_{(i)}]=(sA)_{(0)}.$$\end{definition} Note that the Borcherds' formula implies that if $B\in\mathbf{V}$, then $B_{(0)}$ is a derivation of $\mathbf{V}$. 

\begin{definition} A differential graded VOA is a graded VOA $\mathbf{V}$ equipped with a $\emph{homological}$ differential, $\delta$, which is a derivation of $\mathbf{V}$. We further demand that the vacuum and Virasoro vectors are cycles. \end{definition}

\begin{remark} With the above definition, the cohomology $H(\mathbf{V})$ is a graded VOA.\end{remark}

\begin{definition} If $\mathbf{V}$ is a graded VOA, and $A\in\mathbf{V}$ a degree $-1$ Virasoro primary of conformal weight $1$, satisfying $$A(z)A(w)\sim 0,$$ then the \emph{BRST reduction} of $\mathbf{V}$ with respect to $A$ is the differential graded VOA $$\operatorname{BRST}(\mathbf{V},A):=(\mathbf{V},\partial^{BRST}_{A}:=A_{(0)}).$$\end{definition}

\begin{remark}As remarked above, the differential $\partial^{BRST}_{A}$ is automatically a derivation. Self locality of the field $A(z)$ implies that the differential squares to zero and lemma 2.1. above implies $\partial^{BRST}_{A}L=0$.\end{remark}

\subsection{The Chiral de Rham Complex and topological vertex algebras} This subsection will be particularly devoid of detail, the reader is encouraged to consult the original reference \cite{MSV} for a proper discussion. 

\begin{definition} The \emph{topological vertex algebra}, $\mathbf{top}\{D\}$, of \emph{rank} $D$ is the super vertex algebra generated by vectors $\{L,J,Q,G\}$ such that $L$ and $J$ are even and $Q$ and $G$ are odd. These currents are subject to OPEs the reader can find in \cite{MSV}, subsection 2.1. \end{definition}

\begin{remark} For reference we include some facts about $\mathbf{top}\{D\}$, as we will make frequent use of them below. We have; \begin{itemize}\item $L$ generates a copy of the centreless Virasoro, \item $G$ and $Q$ are conformal primaries of weights $2$ and $1$ respectively, \item the even current $J(z)$ satisfies $$J(z)J(w)\sim\frac{D}{(z-w)^{2}},$$ so that $J$ generates a representation of the Heisenberg at level $D$, \item the odd currents satisfy the OPE $$Q(z)G(w)\sim\frac{L(w)}{(z-w)}+\frac{J(w)}{(z-w)^{2}}+\frac{D}{(z-w)^{3}}$$ and so in particular we have $[G_{(1)},Q_{(0)}]=L_{(1)}$. \end{itemize}\end{remark}

\begin{definition} If $\mathbf{V}$ is a super VOA, we say it is endowed with a \emph{topological structure} of rank $D$ if there is given a distinguished morphism $\mathbf{top}\{D\}\rightarrow \mathbf{V}$, and call it a topological vertex algebra of rank $D$.\end{definition} 

\begin{remark}We will denote the corresponding vectors in $\mathbf{V}$, abusively, as $\{L,J,G,Q\}$.\end{remark}

\begin{theorem} (\cite{MSV}) For $(X,\mathbf{vol})$ an oriented $D$-dimensional smooth variety there is a sheaf in the \'{e}tale topology on $X$ consisting of topological vertex algebras of rank $D$. This sheaf is denoted $\Omega^{ch}_{X}$ and referred to as the \emph{Chiral de Rham Complex}.\end{theorem}

Formally locally on $X$, we may choose coordinates $\{x^{i}\}_{i}$, and then $\Omega^{ch}_{X}$ is generated by fields $$x^{i}(z), y^{i}(z), \phi^{i}(z), \psi^{i}(z)\,; i=1,2,...,D.$$  Here $x$ and $y$ are even and $\phi$ and $\psi$ odd. The only non trivial OPEs are $$y^{i}(z)x^{j}(w)\sim\frac{\delta_{ij}}{(z-w)},\,\,\,\phi^{i}(z)\psi^{j}(w)\sim\frac{\delta_{ij}}{(z-w)}.$$ Note that our $\phi$ and $\psi$ are as in \cite{MSV}, but what we are calling $x$ and $y$ are denoted respectively as $b$ and $a$  in \emph{loc. cit.} 

In terms of the above local generators of $\Omega^{ch}_{X}$, the vectors $\{L,J,G,Q\}$ take the following form, where we sum over repeated indices throughout; \begin{itemize}\item $ L=\delta x^{i}y^{i}+\delta \phi^{i}\psi^{i}$
\item $J=\phi^{i}\psi^{i}$
\item $Q=y^{i}\phi^{i}$
\item $G=\delta x^{i}\psi^{i}$\end{itemize}

\section{Topological structure on the Chiral Critical Locus}

\subsection{Chiral Polyvectors} We will in fact find it more convenient to twist the above topological structure somewhat, doing so will ensure that we can perform BRST reductions in the sense we want. In the \emph{classical} (which is to say \emph{non-chiral}) setting this is already required. Indeed, $\Omega_{X}$ is not a differential graded algebra when endowed with the differential $df\wedge(-)$, whereas $\Theta_{X}$ is a differential graded algebra when endowed with the differential $\{f,-\}$. It is precisely the algebra of functions on the derived critical locus $\mathbf{crit}_{f}$. The twist we will consider will come from the \emph{mirror involution} of $\mathbf{top}\{D\}$, to be defined below. It will have the effect that, whereas the conformal weight $0$ subspace of $\Omega^{ch}_{X}$ is the algebra of differential forms, after the twist the conformal weight zero subspace will be the algebra of polyvectors. 

\begin{definition}The \emph{mirror involution} is defined to be the involution of $\mathbf{top}\{D\}$ defined by $$G\mapsto Q,\, Q\mapsto G,\, J\mapsto -J, L\,\mapsto L-\delta J.$$ \end{definition}

\begin{remark} This is an involution of $\mathbf{top}\{D\}$ considered as a super vertex algebra, not as a super VOA - $L$ is not preserved. \end{remark}

\begin{definition} The topological vertex algebra of \emph{chiral polyvectors}, denoted $\Theta^{ch}_{X}$, is defined to be $\Omega^{ch}_{X}$ as a sheaf of vertex algebras, with topological structure arising by twisting the topological structure on $\Omega^{ch}_{X}$ by the mirror involution. \end{definition}

For the reader's convenience we include formulae for the vectors generating the topological structure on $\Theta^{ch}_{X}$ in terms of the the local generating vectors $x^{i},y^{i},\phi^{i},\psi^{i}$. \\ \\We have; \begin{itemize}\item $ L=\delta x^{i}y^{i}+ \phi^{i}\delta\psi^{i}$
\item $J=-\phi^{i}\psi^{i}$
\item $G=y^{i}\phi^{i}$
\item $Q=\delta x^{i}\psi^{i},$\end{itemize} where we draw the reader's attention to the presence of a derivative of a $\psi$ field as opposed to a $\phi$ field in our new formula for $L$. This has the effect of making $\psi$ a field of conformal degree $0$ now, and indeed we have $$\Theta^{\operatorname{ch},(0)}_{X}=\bigwedge\Theta_{X},$$ the algebra of polyvectors.

\subsection{The Chiral Critical Locus}

We now properly introduce our main object of study, the chiral critical locus $\mathbf{crit}^{ch}_{f}$. This will \emph{chiralize} the usual critical locus, in the sense that the weight zero subspace  of the chiral critical locus will be the sheaf of functions on the usual derived critical locus, $\mathbf{crit}_{f}$. We remind the reader that this is defined as the derived intersection by $$\mathbf{crit}_{f}=X\times^{\mathbf{R}}_{T^{*}X}X,$$ where the copies of $X$ map to $T^{*}X$ by $0$ and $df$ respectively, and $\times^{\mathbf{R}}$ denotes the fibre product of derived schemes. The sheaf of functions on $\mathbf{crit}_{f}$ is computed explicitly as the sheaf of polyvectors on $X$ equipped with the differential $$\iota_{df}=\{f,-\},$$ where $\iota$ denotes contraction and $\{-,-\}$ the bracket of polyvectors. This is an example of a $(-1)$-symplectic variety, the reader is referred to \cite{PTVV} for more details.

As mentioned in the introduction, we will produce $\mathbf{crit}^{ch}_{f}$ from $\Theta^{ch}_{X}$ by BRST reduction. Throughout the rest of this note we will identify $\bigwedge\Theta_{X}$ with the conformal weight zero subspace of $\Theta^{ch}_{X}$. As an example, for a function $f$ on $X$ we will write $f\in\Theta^{ch}_{X}$. We consider $\Theta^{ch}_{X}$ as a graded vertex algebra, with grading given by the operator $J_{(0)}$.

\begin{lemma} The element $G_{(0)}f$ is a conformal weight one primary of degree $-1$.\end{lemma}

\begin{proof} This is immediate from the fact that $L_{(1)}f=0$ and the fact that $G$ is a weight $2$ conformal primary. \end{proof}

\begin{definition} We define the \emph{chiral critical locus} as $$\mathbf{crit}^{ch}_{f}=\operatorname{BRST}(\Theta^{ch}_{X},G_{(0)}f).$$We denote the homological differential on $\mathbf{crit}^{ch}_{f}$ by $\partial^{ch}_{f}$. We consider this as a differential graded VOA with Virasoro  $L=\delta x^{i}y^{i}+\phi^{i}\delta\psi^{i}.$\end{definition}

Note that $\partial^{ch}_{f}$ is a differential because of the OPE $G(z)G(w)\sim 0$ and $L$ is $\partial^{ch}_{f}$-closed Virasoro vector by lemma 2.1.  Further, it is clear enough from looking at the OPEs in local coordinates that this should be interpreted as \emph{chiral differential operators}, in the sense of \cite{GMS}, on the derived critical locus.

\begin{example}We collect some expected properties of this construction, all of which can easily be proven. \begin{itemize}\item If $f$ is smooth then $\mathbf{crit}^{ch}_{f}$ is quasi-isomorphic to $0$. \item Let $q:V\rightarrow\mathbf{A}^{1}$ be a non-degenerate quadratic form on a vector space $V$. Then we find that $\mathbf{crit}^{ch}_{q}$ is quasi-isomorphic to $k$.\item More generally, if $q:\mathcal{Q}\rightarrow\mathbf{A}^{1}$ is the quadratic form on the total space of a quadratic bundle over a space $X$, then the natural map $\Theta^{ch}_{X}\rightarrow\mathbf{crit}^{ch}_{q}$ is a quasi-isomorphism. \item If $f\oplus g$ is the direct sum potential on $X\times Y$, then there is a natural equivalence$$\mathbf{crit}^{ch}_{f\oplus g}\cong\mathbf{crit}^{ch}_{f}\otimes\mathbf{crit}^{ch}_{g}.$$\item If the critical locus of $f$ is proper then the hypercohomology of $\mathbf{crit}^{ch}_{f}$ has finite bi-graded character, where we grade by eigenvalues of $L_{(1)}$ and $J_{(0)}$.\end{itemize}\end{example}

\begin{lemma} The differential $\partial^{ch}_{f}$ restricts to $\{f,-\}$ on the conformal weight $0$ subspace, so that the conformal weight zero subspace is precisely the differential graded algebra $\mathcal{O}(\mathbf{crit}_{f})$. \end{lemma}

\begin{proof} We work locally with respect to our usual coordinates. Note that we have $\{f,-\}=\partial_{j}f\frac{\partial}{\partial\psi^{j}}.$ Both $\{f,-\}$ and $(G_{(0})f)_{(0)}$ are derivations, so we must only prove that $$\partial^{ch}_{f}(x^{i})=0,\,\,\partial^{ch}_{f}\psi^{i}=\partial_{i}f.$$ Indeed, we have $G_{(0)}f=\partial_{j}f\phi^{j}$. The $0$-mode of this obviously kills all the $x^{i}$ and acts on $\psi^{i}$ as $(\partial_{j}f)_{(-1)}\phi^{j}_{(0)}=(\partial_{j}f)\frac{\partial}{\partial\psi^{j}},$ whence we are done. \end{proof}

\begin{lemma} The operator $G_{(1)}$, which is of conformal weight $0$ and so acts on $\bigwedge\Theta_{X}$, is the divergence operator $\Delta_{\mathbf{vol}}$ corresponding to the volume form $\mathbf{vol}$.\end{lemma}

\begin{proof} We work locally, so that $G=\phi^{i}y^{i}$. Then we observe that $y^{i}_{(n)}$ and $\phi^{i}_{(n)}$ act as zero for any $n>0$ so that the Borcherds' formula implies that $G_{(1)}$ acts on $\Theta_{X}$ as $\phi^{i}_{(0)}y^{i}_{(0)}=\frac{\partial}{\,\partial x^{i}}\frac{\partial}{\partial\psi^{i}}\,$ which is precisely the divergence operator expressed with respect to our chosen coordinates. \end{proof}

\begin{lemma} Let $\xi$ be a vector field on $X$, considered as an element of $\Theta^{ch}_{X}$, then we have $$G_{(1)}\xi=\mathbf{vol}^{-1}\operatorname{Lie}_{\xi}(\mathbf{vol}),$$ so that in particular if $\mathbf{vol}$ is an eigenvector of $\operatorname{Lie}_{\xi}$ with eigenvalue $b$, then we have $G_{(1)}\xi=b.$ \end{lemma}

\begin{proof} This is a simple consequence of the Cartan formula $[d,\iota_{\xi}]=\operatorname{Lie}_{\xi}.$ Indeed, we have $$\Delta_{\mathbf{vol}}\xi=\mathbf{vol}^{-1}d(\mathbf{vol}(\xi))=\mathbf{vol}^{-1}d(\iota_{\xi}\mathbf{vol})=\mathbf{vol}^{-1}\operatorname{Lie}_{\xi}\mathbf{vol},$$ where we have used the Cartan formula and the fact $d\mathbf{vol}=0$. The second part of the lemma is of course a trivial consequence of the first.\end{proof}

\begin{lemma} We have the following compatibilities between $\partial^{ch}_{f}$ and the generators of the $\mathbf{top}\{D\}$ action. $$\partial^{ch}_{f}L=\partial^{ch}_{f}G=0,$$ $$\partial^{ch}_{f}J=-G_{(0)}f,$$ $$\partial^{ch}_{f}Q=\delta f.$$ \end{lemma}

\begin{proof}We proceed one by one, working in local coordinates so that $\partial^{ch}_{f}=\partial_{j}f\phi^{j}$. The computations are made extremely easy noting that $\partial^{ch}_{f}$ is a derivation for all products, in particular for the $(-1)$-product.\begin{itemize}\item For $L$, this follows from the fact that $G_{(0)}f$ is a conformal weight one primary as explained above

 \item  Now consider $G=y^{i}\phi^{i}$. Then it is clear that $\partial^{ch}_{f}\phi^{j}=0$ and that $\partial^{ch}_{f}(y^{j})=\partial_{ij}(f)\phi^{i}$ for all $j$. As $\partial^{ch}_{f}$ is a derivation we have $\partial^{ch}_{f}G=\partial^{ch}_{f}(y^{i})\phi^{i}=\partial_{ji}(f) \phi^{i}\phi^{j}=0.$

\item For $J=-\phi^{i}\psi^{i}$ we note that $\partial^{ch}_{f}$ kills $\phi^{i}$. Noting that $\partial^{ch}_{f}(\psi^{j})=\partial_{j}f$, we are done.

\item The case of $Q=\delta x^{i}\psi^{i}$ follows from the fact $\delta f=\partial_{i}f\delta x^{i}$.\end{itemize}\end{proof}

\begin{lemma} (\cite{MSV}) Let $\xi$ be a vector field on $X$. Then $G_{(0)}\xi(z)$ is a self local field. \end{lemma}

\begin{proof} In fact, writing $\xi^{ch}:=G_{(0)}\xi$, the authors show that we have the OPE $$\xi^{ch}(z)\chi^{ch}(w)\sim\frac{[\xi,\chi]^{ch}(w)}{(z-w)}.$$ This an infinitesimal version of the main construction of \emph{loc. cit.} which canonically associates vertex automorphisms to coordinate transformations on $X$. The case where $\xi=\chi$ is an obvious consequence. \end{proof}

We will henceforth assume that $X$ carries an action of $\mathbf{G}_{m}$ for which $\mathbf{vol}$ has weight $b$ and $f$ has weight $a$. \begin{example} We can take an arbitrary smooth projective $Y$ and let $X=K_{Y}$ be the total space of the canonical bundle on $Y$. This is endowed with a $\mathbf{G}_{m}$ action scaling the fibres, and a volume form of degree $1$. We can then take our function $f$ to be some pluri-anticanonical section, $$s\in\Gamma(Y,\omega^{\otimes -a}_{Y})=\Gamma\big(K_{Y},\mathcal{O}\big)^{(a)}.$$ More generally, we can incorporate an action of $\mathbf{G}_{m}$ on $X$ into the above.\end{example}

\begin{definition} For our given $\mathbf{G}_{m}$ action on $X$, we denote by $\xi$ the vector field on $X$ obtained by differentiation. \end{definition}

\begin{theorem} Let $f$ be homogeneous of weight $a$ for a $\mathbf{G}_{m}$ action on $(X,\mathbf{vol})$, with $\mathbf{vol}$ of weight b. Then the currents, $$^{f}L=L,$$ $$^{f}G=G,$$ $$^{f}J=J+\frac{1}{a}G_{(0)}\xi,$$ $$^{f}Q=Q-\frac{1}{a}\delta\xi,$$ are $\partial^{ch}_{f}$-closed and endow $\mathbf{crit}^{ch}_{f}$ with a topological structure at the rank $$d:=\Big(D-\frac{2b}{a}\Big).$$  \end{theorem}

\begin{proof} We first prove closedness with respect to $\partial^{ch}_{f}$. Lemma 3.5 tells us that $^{f}=L$ and $^{f}G=G$ are already closed and that to check closedness of $^{f}J$ and $^{f}Q$ we must verify the relations $$\partial^{ch}_{f}G_{(0)}\xi=-aG_{(0)}f, $$ $$\partial^{ch}_{f}(\delta\xi)=a\delta f.$$ Both of these relations follow easily from $\partial^{ch}_{f}\xi=af$, which can be checked in local coordinates to be the relation $\operatorname{Lie}_{\xi}f=af$, expressing homogeneity of $f$ with respect to the $\mathbf{G}_{m}$ action. 

We will not verify all of the OPEs required to show that the above currents generate an action as it would be rather tedious and at times routine - for example note that it is easy to see that the odd currents $^{f}G$ and $^{f}Q$ are Virasoro primaries of weights $2$ and $1$ respectively, and that they are of degrees $-1$ and $+1$ respectively with respect to $^{f}J_{(0)}$.

We choose instead to compute the charge of the current $^{f}J$, whence the rank $d$ of the putative topological structure, and to verify that this is consistent with the OPE between $^{f}Q$ and $^{f}G$.

First we compute the charge of $^{f}J$, indeed we have $^{f}J:=J+\frac{1}{a}G_{(0)}\xi$. We have $J_{(0)}G_{(0)}\xi=0$ and lemma 3.6 tells us that $G_{(0)}\xi$ is self local. Further we have $J_{(1)}J=D$ by \cite{MSV}. Now we compute $J_{(1)}G_{(0)}\xi=[J_{(1)},G_{(0)}]=-G_{(1)}\xi=-b.$ Similarly we have $(G_{(0)}\xi)_{(0)}J=-b.$ Putting the above together we deduce the OPE $$^{f}J(z)\,^{f}J(w)\sim \frac{d}{(z-w)^{2}}\,\,;\,\,d=D-\frac{2b}{a}.$$

Let us now check that we have the OPE $$^{f}Q(z)G(w)\sim\frac{L(w)}{(z-w)}+\frac{^{f}J(w)}{(z-w)^{2}}+\frac{d}{(z-w)^{3}}.$$ First let us note that as $^{f}Q$ and $Q$ differ by an element in the image of $\delta$ we have not changed the zero mode. As such, we start with the term in the OPE of order $2$ along $z=w$.

We have $Q_{(1)}G=J$ by \cite{MSV}, and so we must show that $$\frac{1}{a}(-\delta\xi)_{(1)}G=\frac{1}{a}G_{(0)}\xi.$$ This follows from $(\delta\xi)_{(1)}=-\xi_{(0)}$ and quasi-symmetry of OPEs by noting, crucially, that $\delta(G_{(1)}\xi)=\delta(b)=0$ from our assumption that the volume form is homogeneous. Finally we will check that we obtain $d$ as the order $3$ term along $z=w$ in the OPE. Indeed this follows from an argument exactly as above, where we pick up a factor of $2$ in front of $b$ now from translation equivariance of OPEs; $(\delta\xi)_{(2)}=-2\xi_{(1)}.$

\end{proof}

\begin{remark} Prior even to the rigorous mathematical introduction of $\Omega^{ch}_{X}$ in \cite{MSV}, the above is essentially proven in \cite{Witt}, albeit in more physical language, for the case of $f=x^{a}$ on $\mathbf{A}^{1}$ as well as for products of this example. In this case the we always have $b=D$ and we obtain $d=D(1-\frac{2}{a}\,)$ as in \emph{loc. cit.}, where the rank arises most naturally from $$\Big(1-\frac{1}{a}\,\Big)^{2}-\Big(\,\frac{1}{a}\,\Big)^{2}=1-\frac{2}{a}.$$\end{remark}

\section{Additional Remarks and Computations}

\subsection{Characters} Recall that if $V$ is a representation of $\mathbf{top}\{d\}$, we define the character by  $$\operatorname{ch}_{q;z}\{V\}:=\operatorname{tr}(q^{L_{(1)}}(-z)^{-J_{(0)}}),$$ where we caution that this is a slightly non-standard normalization for characters. If $\mathcal{V}$ is a sheaf of representations of $\mathbf{top}\{d\}$ then we define the $\operatorname{ch}_{q;z}(\mathcal{V})$ as the character of the hypercohomology of $\mathcal{V}$. Let us see how this can be computed in the the case of $X=K_{Y}$, where $Y$ is a smooth projective variety of dimension $D-1$ admitting a $\mathbf{G}_{m}$ action with isolated fixed point set $\operatorname{Fix}$, and $f$ is a function of weight $a$.

 We will set $$\theta(z;q):=\theta(z):=\prod_{i=0}^{\infty}(1-q^{i}z)(1-q^{i+1}z^{-1})(1-q^{i+1})$$ as usual. \begin{definition} Given a finite set of parameters $\underline{\alpha}:=\{\alpha_{i}\}_{i}$, we define $$\Theta\{\underline{\alpha}\}:=\Theta\{\underline{\alpha}\}(z;q):=\prod_{i}\frac{\theta(z^{1+\alpha_{i}})}{\theta(z^{\alpha_{i}})}.$$ If $\underline{\alpha}$ is a singleton then we write $\Theta\{\alpha\}:=\Theta\{\underline{\alpha}\}$. \end{definition}

 For each $y\in F$, let  $w_{j}(y)$ be the weights of the $\mathbf{G}_{m}$-action on $T_{y}Y$ $\underline{w}(y)$. Now define $\underline{w}(y)$ to be the tuple $$\underline{w}(y):=\frac{1}{a}\{w_{1}(y),...,w_{d}(y)\}.$$ Finally let $$w_{\operatorname{tot}}(y):=\frac{1}{a}\sum_{i}w_{i}(y).$$

We state the following simple computation as a theorem.

\begin{theorem}  $$\operatorname{ch}_{q;z}\{\mathbf{crit}^{ch}_{f}\}=\sum_{y\in\operatorname{Fix}}\Theta\{\underline{w}(y)\}\Theta\{- w_{\operatorname{tot}}(y)\}.$$ \end{theorem}

\begin{proof} We first compute the triply graded character of $\Theta^{ch}_{K_{Y}}$, where we include an extra grading corresponding to the $\mathbf{G}_{m}$ action. This introduces an extra variable $t$ to the character, and is easily computed by the Atiyah Bott localization theorem, cf. \cite{AB}. In fact $\Theta^{ch}_{K_{Y}}$ is not quasi-coherent so we should apply the localization theorem to the associated graded with respect to the filtration defined in 3.10 of \cite{MSV}. Note that $w_{\operatorname{tot}}$ terms arise from the contributions of the fibre direction to the $\mathbf{G}_{m}$-module $T_{y}K_{Y}$.

A routine computation produces $$\sum_{y\in\operatorname{Fix}}\Big(\prod_{i}^{D-1}\frac{\theta(zt^{w_{i}(y)})}{\theta(t^{w_{i}(y)})}\Big)\frac{\theta(zt^{-w_{\operatorname{tot}}(y)})}{\theta(t^{-w_{\operatorname{tot}}(y)})},$$ 

Recalling that $$^{f}J:=J+\frac{1}{a}G_{(0)}\xi,$$ and noting that the $0$-mode of $G_{(0)}\xi$ is the Euler operator corresponding to the grading, we see that we must now set $t=z^{1/a}$ to obtain the desired character.\end{proof}

\subsection{BV Quantization}

We will drop the requirement that $X$ come equipped with a $\mathbf{G}_{m}$ action henceforth, so that the discussion below has nothing to do with topological vertex algebras. 

Recall that $\mathbf{crit}_{f}$ has the structure of a $(-1)$ symplectic variety in the sense of \cite{PTVV}. Then, following the discussion in the introduction to \cite{Co}, we can speak of \emph{BV quantizations} of $\mathbf{crit}_{f}$. Further, a natural such is provided by a volume form on $X$, in this case $\Delta_{\mathbf{vol}}$ is a quantization. Recall that this means that $\Delta_{\mathbf{vol}}$ is a second order differential operator whose failure to be a derivationproduces the bracket on $\mathcal{O}(\mathbf{crit}_{f})$ dual to the shifted symplectic form. Associated to this quantization is the \emph{BV complex}, defined as $$\mathbf{BV}_{f}:=\big(\bigwedge\Theta_{X}((\hbar)),\{f,-\}+\hbar\Delta_{\mathbf{vol}}).$$ Now, there is no problem in defining a chiral version of this, $$\mathbf{BV}^{ch}_{f}:=(\Theta^{ch}_{X}((\hbar)),\partial^{ch}_{f}+\hbar G_{(1)}).$$ This is graded by conformal weight in the obvious way and the weight $0$ subspace is precisely $\mathbf{BV}_{f}$. \begin{remark} In the chiral case one can consider $G_{(1)}$ as quantizing what we referred to in \cite{Bo} as a \emph{Gerstenhaber vertex algebra}. \end{remark}  

We might say that $\mathbf{crit}^{ch}_{f}$ is a \emph{Hodge} type object, and $\mathbf{BV}^{ch}_{f}$ is a \emph{de Rham} analogue of it. As such, one one might expect, by analogy with the case of the chiral de Rham complex endowed with its chiral de Rham differential, that $\mathbf{BV}^{ch}_{f}$ does not differ homologically from its non chiral counterpart $\mathbf{BV}^{ch}_{f}$. This is indeed the case, we have the following simple theorem.

\begin{theorem} The inclusion $\mathbf{BV}_{f}\rightarrow\mathbf{BV}^{ch}_{f}$ is a quasi-isomorphism. \end{theorem}

\begin{proof} We consider the evident $\hbar$-adic spectral sequence computing homology of $\mathbf{BV}^{ch}_{f}$. Then the $E^{1}$ page is $H(\mathbf{crit}^{ch}_{f})((\hbar))$ with the differential $\hbar G_{(1)}$. Now let us note that $Q_{(0)}$ descends to an operator on $H(\mathbf{crit}^{ch}_{f}).$ Indeed we have seen in lemma 3.5. that $\partial^{ch}_{f}(Q)=\delta f$, whence we have $$[\partial^{ch}_{f},Q_{(0)}]=(\delta f)_{(0)}=0.$$ Now the result follows immediately from $[Q_{(0)},G_{(1)}]=L_{(1)}$. \end{proof}

\begin{remark}Recalling that the exponentially twisted de Rham complex can be computed as the formal Laplace transform of the Gauss Manin system associated to $f$, it would be interesting to compare the above with the \emph{chiral} Gauss Manin connection, recently introduced in \cite{MST}. \end{remark}

\end{document}